\documentclass[10pt]{amsart}

\usepackage{psfrag}
\usepackage{amsmath}
\usepackage{amssymb}
\usepackage[english]{babel}
\usepackage{amsthm}
\usepackage{amsfonts}
\usepackage[applemac]{inputenc}
\usepackage{tikz}
\usetikzlibrary{arrows}
\usetikzlibrary{patterns}
\usetikzlibrary{shapes}
\usepackage{mathrsfs}
\usepackage{bbm}

\newtheorem{teo}{Theorem}
\newtheorem{lemma}[teo]{Lemma}

\newtheorem{coro}[teo]{Corollary}

\newcommand{\R}{\mathbb{R}}

\theoremstyle{remark}

\title[]{Plenty of Morse functions by perturbing with sums of squares}

\author{ A. Lerario}
\thanks{SISSA, Trieste}

\begin{document}

\maketitle

\begin{abstract} We prove that given a smooth function $f:\R^{n}\to \R$ and a submanifold $M\subset \R^{n},$ then the set of $a=(a_{1},\ldots, a_{n})\in \R^n$ such that $(f+q_{a})|_{M}$ is Morse, where $q_{a}(x)=a_{1}x_{1}^{2}+\cdots +a_{n}x_{n}^{2},$ is a residual subset of $\R^{n}.$ A standard transversality argument seems not to work and we need a more refined approach.
\end{abstract}

\section{Introduction} The following problem was posed to the author by K. Kurdyka, to which we express our gratitude for stimulating discussions.\\
Given $a=(a_{1},\ldots, a_{n})\in \R^{n}$ we define the function $q_{a}:\R^{n}\to \R$ by
$$q_{a}(x)=a_{1}x_{1}^{2}+\cdots +a_{n}x_{n}^{2}$$
where in the case $b\in \R^{k}, k\neq n,$ we mean $q_{b}$ to belong to $C^{\infty}(\R^{k},\R)$ and to be defined in the same similar way.\\
Suppose $f\in C^{\infty }(\R^{n},\R)$ and $M\subset \R^{n}$ is a submanifold. Is it true that the set
$$A(f,M)=\{a\in \R^{n}\,|\, (f+q_{a})|_{M}\,\,\textrm{is Morse}\}$$
is residual in $\R^{n}?$\\
The answer to this question turns out to be affirmative, but in a subtle way: standard transversality arguments based on dimension counting do not work and we have to prove it directly.

\section{Failure of parametric transversality argument}
We describe here what it is the usual procedure to prove that given a family of functions $f_{a}:M\to\R$ depending smoothly on the parameter $a\in A$ then the set of $a$ such that $f_{a}$ is Morse is residual in $A.$\\
Let $G:A\times M\to N$ be a smooth map and for every $a\in A$ let $g_{a}:M\to N$ be the function defined by $x\mapsto G(a,x).$ Suppose that $Z\subset N$ is a submanifold and that $F$ is transverse to $Z.$ Then from the \emph{parametric transversality theorem} (see \cite{Hirsch}, Theorem 2.7) it follows that $\{a\in A\,|\, g_{a}\,\, \textrm{is transverse to $Z$}\}$ is residual in $A.$\\
 In the case we want to get Morse condition consider $N=T^{*}M,$ $G(a,x)=d_{x}f_{a}$ and $Z\subset T^{*}M$ the zero section. Then $f_{a}$ is Morse if and only if $g_{a}$ is transverse to $Z.$\\
In our case, letting $M=\R^{n}$, we are led to define $G:\R^{n}\times \R^{n}\to \R^{n}$ by
$$(a,x)\mapsto (\partial f/\partial x_{1}(x)+a_{1}x_{1},\ldots, \partial f/\partial x_{n}(x)+a_{1}x_{n})$$
and we consider $Z=\{0\}\in \R^{n}.$ Then $f_{a}=f+q_{a}$ is Morse if and only if $g_{a}$ is transversal to $\{0\}.$ A condition that would ensure this (trough the parametric transversality theorem) is that $G$ is transverse to $\{0\}.$ Computing the differential of $G$ at the point $(a,x)$ we have for $(v,w)\in T_{(a,x)}(\R^{n}\times \R^{n})$
$$(d_{(a,x)}G)(v,w)=\textrm{He}(f)(x)v+\textrm{diag}(a_{1},\ldots, a_{n})v+\textrm{diag}(x_{1},\ldots, x_{n})w$$ and we see that in general this condition does not hold (for example let $f\equiv 0,$ then at alle the points $(a,x)=(0,a_{2},\ldots, a_{n},0,\ldots, 0)$ we have $G(a,x)=0$ but $\textrm{rk}(d_{(a,x)}G)<n$). 
\section{A direct approach}

First we recall the following Lemma (see \cite{Bott}).
\begin{lemma}Let $f$ be a smooth function on $\R^{n}$ and for $a\in\R^{n}$ define the function $f_{a}$ by $x\mapsto f(x)+a_{1}x_{1}+\ldots +a_{n}x_{n}.$  The set $$\{a\in\R^{n}\,|\, f_{a}\,\,\textrm{is Morse}\}$$ is residual in $\R^{n}.$ \end{lemma}
\begin{proof}
Define the function $g(x)=(\partial f/\partial x_{1},\ldots, \partial f/\partial x_{n})$ and notice that the Hessian of $f$ is precisely the Jacobian of $g$ and that $x$ is a nondegenerate critical point for $f$ if and only if $g(x)=0$ and the Jacobian $J(g)(x)$ of $g$ at $x$ is nonsingular. Then $g_{a}(x)=g(x)+a$ and $J(g_{a})=J(g).$ We have that $x$ is a critical point for $f_{a}$ if and only if $g(x)=-a;$ moreover  it is a nondegenerate critical point if and only if we also have $J(g)(x)$ is nonsingular, i.e. $a$ is a regular value of $g.$ The conlusion follows by Sard's lemma. 
\end{proof}

We immediately get the following corollary.
\begin{coro}\label{nonz}If $f$ is a smooth function on an open subset $U$ of $\R^{n}$ such that for every $u=(u_{1},\ldots, u_{n})\in U$ we have $u_{i}\neq 0$ for all $i=0,\ldots ,n,$ then 
$$A(f,U)=\{a\in \R^{n}\, |\, f+q_{a} \,\,\textrm{is Morse on $U$}\}$$ is a residual subset of $\R^{n}.$
\end{coro}
\begin{proof}The functions $u_{1}^{2},\ldots,u_{n}^{2}$ are coordinates on $U$ by hypothesis; we let $\tilde f$ be the function $f$ in these coordinates (it is defined on a certain open subset $W$ of $\R^{n}$). Then for every $a\in \R^{n}$ we have that (using the above notation) $\tilde f_{a}$ is Morse on $W$ if and only if $f+q_{a}$ is Morse on $U$ and the conclusion follows applying the previous lemma.
\end{proof}

To prove the general statement we need the following.

\begin{lemma}\label{open}Let $f$ be a smooth function on an (arbitrary) open subset $U$ of $\R^n.$ Then the set $A(f,U)$ is residual in $\R^{n}$.
\begin{proof}
For every $I=\{i_1,\ldots, i_j\}\subset \{1,\ldots,n\}$ define $$H_I=U\cap \{ u_{i}=0,\, i\in I\}\cap\{u_{k}\neq 0,\, k\notin I\}.$$
To simplify notations let $I=\{1,\ldots, j\}$. Notice that if $a=(a_1,\ldots,a_n)$ and $a''=(a_{j+1},\ldots,a_n)$ then $(q_a)|_{H_{I}}=(q_{a''})|_{H_{I}}$ where $q_{a''}:\R^{n-j}\to \R$ is defined as above.  
By corollary \ref{nonz} the set $$A''(f, H_{I} )=\{a''\in \R^{n-j}\, |\, \textrm{$f|_{H_{I}}+q_{a''}$ is Morse on $H_I$}\}$$ is residual in $\R^{n-j}$.
Let $a=(a',a'')\in \R^n$ such that $a''\in A''(f, H_{I} )$ and suppose  $x\in H_I$ is a critical point of $f+q_a;$ then $x$ is also a critical point of $(f+q_a)|_{H_I}=f|_{H_{I}}+q_{a''}.$ Since $a''\in  A''(f, H_{I} )$ then $x$ belongs to a countable set, namely the set $C_{a''}$ of critical points of $ f|_{H_{I}}+q_{a''}$ (each of this critical point must be nondegenerate by the choice of $a''$); moreover we have that $$\textrm{He}(f|_{H_{I}}+q_{a''})(x)=\textrm{He}(f|_{H_{I}})(x)+\textrm{diag}(a_{j+1},\ldots,a_n)$$ is nondegenerate.
Notice that the Hessian of $f+q_a$ at $x$ is a block matrix: $$\textrm{He}(f+q_a)(x)= \left(\begin{array}{c|c}
\textrm{diag}(a_1,\ldots,a_j)+B(x) & C(x) \\ \hline C(x)^{T}& \textrm{He}(f|_{H_{I}}+q_{a''})(x)\end{array}\right). $$
Thus for every $a''=(a_{j+1},\ldots,a_{n})\in A''(f, H_{I})$ and for every $x\in C_{a''}$ consider the polynomial $p_{a'',x}\in \R[t_1,\ldots,t_j]$ defined by $$p_{a'',x}(t_1,\ldots,t_j)=\textrm{det}(\textrm{He}(f)(x)+\textrm{diag}(t_1,\ldots,t_j,a_{j+1},\ldots,a_n))$$
Then the term of maximum degree of $p_{a'',x}$ is $$t_{1}\cdots t_{j}\det (\textrm{He}(f|_{H_{I}}+q_{a'})(x))$$ which is nonzero since $\textrm{det}(\textrm{He}(f|_{H_{I}}+q_{a''})(x))\neq 0$ ($x$ is a \emph{nondegenerate} critical point of $f|_{H_{I}}+q_{a''}$). It follows that  $p_{a'',x}$ is not identically zero; hence its zero locus is a \emph{proper} algebraic set. Thus  for  each $a''\in A''(f, H_{I})$ and each $x\in C_{a''}$ the set $A'(a'',x, I)$ defined by $$\{a'\in \R^j\, |\, \textrm{if $x$ is a critical point of $f+q_{(a',a'')}$ on $H_I$ then it is nondegenerate}\}$$ is residual in $\R^{j}$ (it is the complement of a proper algebraic set);  it follows that $$A'(a'',I)=\{a'\in \R^j\, |\, \textrm{each critical point of } f+q_{(a',a'')}\, \textrm{on } H_{I}\textrm{ is nondegenerate}\}$$ is residual in $\R^{j}$, since it is a countable intersection of residual sets, i.e. $$A'(a'',I)=\bigcap_{x\in C_{a''}}A'(a'',x,I)$$
Thus the set $$A(f,I)=\{(a',a'')\, |\, a''\in A''(f, H_{I}), \, a'\in A'(a'',I)\}$$ (which coincides with the set of $a=(a',a'')\in \R^n$ such that each critical point of  $f+q_a$ on  $H_{I}$ is nondegenerate) is residual: is residual in $a'$ for every $a''$ belonging to a residual set.
Finally $$A(f,U)=\bigcap_{I\subset \{1,\ldots,n\}}A(f,I)$$ is a finite intersection of residual sets, hence residual.
\end{proof}
\end{lemma}

\begin{teo}\label{Morse}Let $f$ be a smooth function on $\R^{n}$ and $M\subset \R^{n}$ be a submanifold. Then the set $A(f,M)$ is residual in $\R^{n}$.
\end{teo}
\begin{proof}We basically improve the proof of Proposition 17.18 of \cite{Bott}.\\
Let $u_1,\ldots,u_n:\R^{n}\to\R$ be the coordinates on $\R^n.$ Suppose $M$ is of dimension $m.$
For every point $\overline{x}\in M$ there exists a neighborhood $W$ of $\overline{x}$ in $M$ such that $u_{i_1},\ldots, u_{i_{m}}$ are coordinates for $M$ on $$W\simeq \R^{m},$$ for some $\{i_1,\ldots,i_m\}\subseteq \{1,\ldots,n\};$ since $M$ is second countable, then it can be covered by a countable (finite if $M$ is compact) number of such open sets. For convenience of notations suppose $\{i_1,\ldots,i_m\}=\{1,\ldots,m\}.$\\
Thus $u_{1},\ldots, u_{m}$ are coordinates on $W\simeq \R^{m}$ and $f|_{W},u_{m+1}|_{W},\ldots, u_{n}|_{W}$ are functions of $u_{1}|_{W},\ldots,u_{m}|_{W}.$
Fix $a''=(a_{m+1},\ldots, a_{n})\in\R^{n-m}$ and define $g_{a''}:W\to \R$ by $$g_{a''}=f|_{W}+a_{m+1}u_{m+1}^{2}|_{W}+\cdots+a_{n}u_{m}^{2}|_{W}=(f+a_{m+1}u_{m+1}^{2}+\cdots+a_{n}u_{n}^{2})|_{W}$$
Notice that $g_{a''}$ is not $(f+q_{a})|_{W}$ since we are taking only the last $n-m$ of the $a_{i}'s;$ we still have the freedom of choice $(a_{1},\ldots,a_{m}).$ \\
By lemma \ref{open}, since $u_{1}|_{W},\ldots, u_{m}|_{W}$ are coordinates on $W$, for every $a''\in \R^{n-m}$ the set $$\{a'=(a_1,\ldots,a_m)\in \R^{m}\quad \textrm{s.t.}\quad g_{a''}+a_1u_1^{2}|_{W}+\cdots +a_mu_m^{2}|_{W}\textrm{ is Morse on  $W$}\}$$ is residual in $\R^{m}$. 
Notice that $ g_{a''}+a_1u_1^{2}|_{W}+\cdots +a_mu_m^{2}|_{W}=(f+q_{(a',a'')})|_{W};$ hence for every $a''$ the set of $a'$ such that $(f+q_{(a',a'')})|_{W}$ is Morse on $W$ is residual. Thus the set of $a\in \R^n$ such that $(f+q_a)|_{W}$ is Morse on $W$ is residual (it is residual in $a'$ for each fixed $a''$ hence it is globally residual). It follows that $A(f,M)$ is a countable intersection of residual set, hence residual.\end{proof}

\end{document}